\documentclass[11pt,a4paper]{article}
\usepackage{amssymb}
\usepackage{eurosym}
\usepackage{amsfonts}
\usepackage{amsmath}
\usepackage{amsthm}
\usepackage{graphicx}
\usepackage{float}
\usepackage{hyperref}
\usepackage[pagewise]{lineno}

\hypersetup{
	colorlinks=true,
	linkcolor=blue,
	anchorcolor=blue,
	citecolor=blue
}
\newtheorem{theorem}{Theorem}[section]

\newtheorem{corollary}[theorem]{Corollary}
\newtheorem{lemma}[theorem]{Lemma}
\newtheorem{proposition}[theorem]{Proposition}
\theoremstyle{definition}
\newtheorem{definition}[theorem]{Definition}
\newtheorem{example}[theorem]{Example}

\newtheorem{remark}[theorem]{Remark}

\newtheorem*{coi}{Conflict of interest statement}
\newtheorem*{da}{Data availability}
\renewenvironment{proof}[1][Proof]{\noindent\textbf{#1.} }{\ \rule{0.5em}{0.5em}}
\newenvironment{acknowledgement}{\smallskip{\sc Acknowledgement.}\rm}{\smallskip}
\renewcommand{\theequation}{\thesection.\arabic{equation}}
\allowdisplaybreaks
\input{tcilatex}

\def\func#1{\mathop{\mathrm{#1}}\nolimits}

\def\dint{\displaystyle\int}

\def\Xint#1{\mathchoice
{\XXint\displaystyle\textstyle{#1}}%
{\XXint\textstyle\scriptstyle{#1}}%
{\XXint\scriptstyle\scriptscriptstyle{#1}}%
{\XXint\scriptscriptstyle\scriptscriptstyle{#1}}%
\!\int}
\def\XXint#1#2#3{{\setbox0=\hbox{$#1{#2#3}{\int}$ }
\vcenter{\hbox{$#2#3$ }}\kern-.6\wd0}}

\def\oint{\Xint-}

\def\enddoc{\end{document}}

\def\FRAME#1#2#3#4#5#6#7#8
{
 \begin{figure}[H]
 \begin{center}
 \includegraphics[height=#3]{#7}
 \caption{#5}
 \label{#6}
 \end{center}
 \end{figure}
}

\begin{document}
	\title{Relative Faber–Krahn inequalities and Trudinger's equation on Riemannian Manifolds}
	\author{Philipp S\"urig}
	\date{August 2026}
	\maketitle
	
	\begin{abstract}
		We consider on Riemannian manifolds the Trudinger equation \begin{equation*}\label{eqabs}\partial _{t}u=\Delta _{p}u^{\frac{1}{p-1}},\end{equation*} where $p>1$.
		We prove that a relative $p$--Faber–Krahn inequality is equivalent to the conjunction of volume doubling and a sub-Gaussian upper estimate for non-negative bounded weak subsolutions of Trudinger's equation. We also derive an improved long-time upper estimate under a uniform $p$-Faber–Krahn inequality.
	\end{abstract}
	
	\let\thefootnote\relax\footnotetext{\textit{\hskip-0.6truecm 2020 Mathematics Subject Classification.} 35K55, 58J35, 35K92. \newline
		\textit{Key words and phrases.} Trudinger equation, doubly nonlinear
		parabolic equation, Riemannian manifold. \newline
		The author was funded by the Deutsche Forschungsgemeinschaft (DFG,
		German Research Foundation) - Project-ID 317210226 - SFB 1283.}
	
	\tableofcontents
	
	\section{Introduction}
	
	Let $M$ be a Riemannian manifold. We consider solutions of the non-linear evolution
	equation 
	\begin{equation}
		\partial _{t}u=\Delta _{p}u^{\frac{1}{p-1}},  \label{evoeq}
	\end{equation}%
	where $p>1$, $u=u(x,t)$ is an unknown non-negative function of $x\in M$, $t\geq0$ and $%
	\Delta _{p}$ is the Riemannian $p$-Laplacian 
	$\Delta _{p}v=\func{div}\left( |\nabla v|^{p-2}\nabla v\right).$
	For the physical meaning of (\ref{evoeq}) see \cite{grigor2024finite, leibenzon1945general, leibenson1945turbulent}. 
	
	The equation (\ref{evoeq}) is also referred to as \textit{Trudinger's equation} \cite{trudinger1968pointwise} or a \textit{doubly non-linear parabolic equation}. In the case $p=2$, it becomes the classical \textit{heat equation} $\partial _{t}u=\Delta u$.

In the present paper we are interested in the relation between properties of solutions of (\ref{evoeq}) and the geometry of the underlying manifold. 

Let $M$ be a geodesically complete, non-compact Riemannian manifold. Denote by $\mu$ the \textit{Riemannian measure} on $M$, by $d$ the \textit{geodesic distance} and by $B(x, r)$ the \textit{geodesic ball} of radius $r$ centered at $x$.
Let us introduce the following geometric property.

\begin{definition}\label{defrfKint}
We say that $M$ satisfies a \textit{relative Faber-Krahn inequality} of order $p> 1$ if there exist positive constants $c, \nu$ so that, for all geodesic balls $B$ of radius $R$, and any non-negative $w\in {W}_{0}^{1,p}(B)$,
\begin{equation}
\dint_{B}\left\vert \nabla w\right\vert ^{p}\geq \dfrac{c}{R^{p}}\left(\dfrac{\mu (B)}{\mu (D)}\right) ^{\nu
}\dint_{B}w^{p},   \label{FKp}\tag{FK$_{p}$}
\end{equation} where $D=\left\{ w>0\right\}$. 
\end{definition} 

For example, the relative Faber-Krahn inequality holds if $M$ has non-negative \textit{Ricci-curvature} (see \cite{Buser, grigor, Saloff}).
In particular, in this case the value of $\nu$ can be chosen as follows:
\begin{equation}\label{nu}
	\nu =\left\{ 
	\begin{array}{ll}
		\dfrac{p}{n}, & \text{if }n>p, \\ 
		\text{any number}\in (0, 1), & \text{if }n\leq p,%
	\end{array}%
	\right.
\end{equation}    
where $n=\textnormal{dim}~M$.

\begin{definition}
We say that $M$ admits a \textit{sub-Gaussian upper estimate} if, for any non-negative $u_{0}\in L^{1}(M)\cap L^{\infty}(M)$ with compact support, any non-negative bounded subsolution $u$ of (\ref{evoeq}) in $M\times [0, \infty)$ with initial function $u_0$ satisfies, for all $x\in M$ and all $t>0$, \begin{equation}\label{uppergaussint}\tag{UE$_{p}$}||u(\cdot,t)||_{L^{\infty}\left(B(x, \frac{1}{2}t^{1/p})\right)}\leq\frac{C||u_{0}||_{L^{1}(M)}}{\mu(B(x, t^{1/p}))}\exp\left(-c\left(\frac{d(x,A)}{t^{1/p}}\right)^{\frac{p}{p-1}}\right),\end{equation} where $A=\textnormal{supp}~u_{0}$ and $c, C$ are positive constants.	
\end{definition}

One of the main purposes of the present paper is the characterization of the relative Faber-Krahn inequality \eqref{FKp} in terms of the estimate (\ref{uppergaussint}). 
In order to do that, we also need to introduce the following geometric property of a manifold.

\begin{definition}\normalfont We say that $M$ admits the \textit{volume doubling property} if there exists positive constants $C, N$ such that for all $x \in M$ and all $0<r\leq R$, \begin{equation}\tag{VD}\label{lemmaausdoubling}\frac{\mu(B(x, R))}{\mu(B(x, r))}\leq C\left(\frac{R}{r}\right)^{N}.\end{equation}
\end{definition}

In the linear case $p=2$, that is, (\ref{evoeq}) becomes the heat equation, it was proved by Grigor'yan \cite{grigor1994heat} that $$(\textnormal{FK}_2)\Leftrightarrow (\textnormal{VD})+(\textnormal{UE}_2).$$ 

Here we extend this result to the nonlinear setting. The first main result of the present paper is as follows.
\begin{theorem}\label{mainthmint}
Let $p>1$. Then 
$$\eqref{FKp}\Leftrightarrow \eqref{lemmaausdoubling}+\eqref{uppergaussint}.$$
\end{theorem}

Observe that, if $M$ satisfies \eqref{FKp} for some $p=p_0>1$, then \eqref{FKp} is true for all $p>p_0$. Hence, we obtain from Theorem \ref{mainthmint} the following

\begin{corollary}
Assume that \eqref{lemmaausdoubling} holds and that \eqref{uppergaussint} is valid for some $p=p_0>1$. Then \eqref{uppergaussint} holds true for every $p>p_0$.
\end{corollary}

In Theorem \ref{mainthmFK} we prove the implication $\eqref{FKp}\Leftarrow \eqref{lemmaausdoubling}+\eqref{uppergaussint}.$

In Proposition \ref{FKimpVD} we prove that $\eqref{FKp}\Rightarrow \eqref{lemmaausdoubling}$.
The implication $\eqref{FKp}\Rightarrow\eqref{uppergaussint}$ was proved in Theorem 1.1 in \cite{surig2024sharp}. Hence, combining these two results we obtain the implication $\eqref{FKp}\Rightarrow \eqref{lemmaausdoubling}+\eqref{uppergaussint}$.

The main technical difficulties of proving the main result of the present paper Theorem \ref{mainthmFK} arise because of the non-linearity of (\ref{evoeq}). In particular, in contrast to the linear case, no spectral expansion in eigenfunctions is available. However, due to the homogeneity of (\ref{evoeq}) it is still possible to apply the upper estimate (\ref{uppergaussint}) to the function $v(x, t)=e^{-\lambda_{1, p}(D)t}\varphi(x)^{p-1}$, where $D$ is a precompact open set in $M$ and $\varphi$ is the eigenfunction corresponding to the \textit{first eigenvalue} $\lambda_{1, p}(D)$ of the $p$-Laplacian $\Delta _{p}$, that is, $-\Delta _{p}\varphi=\lambda_{1, p}(D)\varphi^{p-1}$ in $D$ (cf. Lemma \ref{lemmaeigen}).
For related properties of the first eigenvalue in the Euclidean setting we refer to \cite{lindgren2022comparison, lindqvist1990equation, lindqvist2008nonlinear}.

To state the second main result of the present paper, we need the following

\begin{definition}
We say that $M$ admits a uniform $p$-Faber-Krahn inequality if there exists a positive (non-increasing) function $\Lambda_{p}$ on $\mathbb{R}_{+}$ such that, for all precompact open sets $\Omega\subset M$ and any non-negative $w\in W_0^{1, p}(\Omega)$,  \begin{equation}\label{uFKp}\int_{\Omega}{|\nabla w|^{p}d\mu}\geq \Lambda_{p}(\mu(\Omega))\int_{\Omega}w^{p}d\mu.\end{equation}
\end{definition} 

The second main result of the present paper is the following (cf. Theorem \ref{uniFKthm}).

\begin{theorem}\label{ThmSobint}
	Assume that $n>p$ and suppose that $M$ admits a uniform $p$-Faber-Krahn inequality with Faber-Krahn function \begin{equation}\label{lowerFKcondint}\Lambda_{p}(v)\geq cv^{-p/n}.\end{equation}
	Let $u$ be a non-negative bounded solution of (\ref{evoeq}) in $M\times \mathbb{R}_{+}$ with $u(\cdot, 0)=u_{0}\in L^{1}(M)\cap L^{\infty}(M)$. Fix some $\sigma\geq \frac{p}{p-1}$. Then, for all $x\in M$ and all $t>0$, \begin{equation}\label{Linfupperint}||u(\cdot,t)||_{L^{\infty}\left(B(x,\frac{1}{2}t^{1/p})\right)}\leq \frac{C||u_{0}||_{L^{1}(M)}}{t^{n/(p\sigma)}\gamma(c\sigma^{-(p-1)}t)^{(\sigma-1)/\sigma}}\exp\left(-c\left(\frac{d(x,A)}{t^{1/p}}\right)^{\frac{p}{p-1}}\right),\end{equation} where the constants $c, C$ depend on $p, n$, and $\gamma$ is defined by \begin{equation}\label{gammaint}t=\int_{0}^{\gamma(t)}\frac{dv}{\Lambda_{p}(v)v}.\end{equation}
\end{theorem}

Under the hypothesis of Theorem \ref{ThmSobint} it was proved in Theorem 1.2 in \cite{surig2024sharp} that \begin{equation}\label{simplees}||u(\cdot,t)||_{L^{\infty}\left(B(x,\frac{1}{2}t^{1/p})\right)}\leq \frac{C||u_{0}||_{L^{1}(M)}}{t^{n/p}}\exp\left(-c\left(\frac{d(x,A)}{t^{1/p}}\right)^{\frac{p}{p-1}}\right),\end{equation} which also follows from (\ref{Linfupperint}) and (\ref{lowerFKcondint}). When $M=\mathbb{R}^{n}$ we have $\Lambda_p(v)=cv^{-p/n}$.

Let us discuss examples of manifolds where Theorem \ref{ThmSobint} gives a better estimate than (\ref{simplees}) (cf. Example \ref{ex2}). 
Suppose that, for all large $v$, \begin{equation*}\label{lowerfk2int}\Lambda_p(v)=c (\log v)^{-\alpha},
\end{equation*} where $\alpha\geq0$. When $M$ is the \textit{hyperbolic space} $\mathbb{H}^{n}$ we have $\alpha=0$. For $\alpha>0$ see \cite{coulhon1993isoperimetrie}. Then, for large $t$,
$$||u(\cdot,t)||_{L^{\infty}\left(B(x,\frac{1}{2}t^{1/p})\right)}\leq \frac{C||u_{0}||_{L^{1}(M)}}{t^{n/(p\sigma)}\exp(c_\sigma t^{1/(\alpha+1)})}\exp\left(-c\left(\frac{d(x,A)}{t^{1/p}}\right)^{\frac{p}{p-1}}\right).$$

In the linear case $p=2$, a similar result to Theorem \ref{ThmSobint} was proved in \cite{grigor1994heat}.

In \cite{surig2026long} it was proved that the manifold satisfying the assumptions of Theorem \ref{ThmSobint}, that is, $M$ admits the uniform Faber-Krahn inequality (\ref{uFKp}) with $\Lambda_p(v)$ such that (\ref{lowerFKcondint}) holds, is equivalent to the upper estimate $$||u(\cdot,t)||_{L^{\infty}\left(M\right)}\leq \frac{C||u_{0}||_{L^{1}(M)}}{t^{n/p}}$$ for non-negative bounded solutions of (\ref{evoeq}) with initial function $u_{0}\in L^{1}(M)\cap L^{\infty}(M)$.

Recently, it was proved in \cite{elenius2026doubling} that on \textit{metric measure spaces} a $p$-\textit{Poincaré-inequality} and the volume doubling property (\ref{lemmaausdoubling}) are equivalent to a \textit{parabolic Harnack inequality} for solutions of (\ref{evoeq}).

We understand solutions in a certain weak sense (see Section \ref{secweak} for the definition). For existence results for solutions of (\ref{evoeq}) on Riemannian manifolds, we refer to \cite{surig2026existence}.

Let us describe the structure of the present paper.

In Section \ref{secweak}, we
define the notion of a weak solution of the Trudinger equation (\ref{evoeq}).

In Section \ref{SecRFK} we prove in Theorem \ref{mainthmFK} the implication $\eqref{FKp}\Leftarrow \eqref{lemmaausdoubling}+\eqref{uppergaussint}$.

In Section \ref{Secini} we prove in Theorem \ref{uniFKthm} the upper estimate (\ref{Linfupperint}) assuming a uniform $p$-Faber-Krahn inequality (\ref{uFKp}) satisfying (\ref{lowerFKcondint}). In Proposition \ref{convuFK} we also prove in some sense a converse statement to Theorem \ref{uniFKthm}. We prove an uniform Faber-Krahn inequality under the assumption that an upper estimate as in (\ref{Linfupperint}) holds, where $\gamma$ and $\Lambda_p$ are related by (\ref{gammaint}).

We denote by $c, C$ positive constants whose value might change at each occurrence. For functions $f$ and $g$ we also use the notation $f\simeq g$ if there exists a positive constant $C$ such that $C^{-1}g\leq f\leq Cg$.

\begin{acknowledgement}
The author would like to thank Alexander Grigor'yan for many helpful suggestions.
\end{acknowledgement}

	\section{Weak subsolutions}
	
	\label{secweak}
	
	We consider in what follows the following non-linear evolution
	equation on a Riemannian manifold $M$:%
	\begin{equation}
		\partial _{t}u=\Delta _{p}u^{\frac{1}{p-1}}.  \label{dtv}
	\end{equation} 
	By a \textit{subsolution} of (\ref{dtv}) we mean a non-negative function $u$
	satisfying $$\partial _{t}u\leq\Delta _{p}u^{\frac{1}{p-1}}$$
	in a certain weak sense as explained below.
	
	We assume throughout that 
	\begin{equation*}
		p>1.
	\end{equation*}%

Let $\mu $ denote the Riemannian measure on $M$. For simplicity of notation,
we frequently omit in integrations the notation of measure. All integration
in $M$ is done with respect to $d\mu $, and in $M\times \mathbb{R}$ -- with
respect to $d\mu dt$, unless otherwise specified.

Let $\Omega$ be an open set in $M$ and $I$ be an interval in $[0, \infty)$.
	
	\begin{definition}
		\normalfont
		We say that a non-negative function $u=u(x, t)$ is a \textit{weak
			subsolution} of (\ref{dtv}) in $\Omega\times I$, if 
		\begin{equation}  \label{defvonsoluq}
			u\in C\left(I; L^{\frac{p}{p-1}}(\Omega)\right)\quad \textnormal{and}\quad 
			u^{\frac{1}{p-1}}\in L_{loc}^{ p}\left(I; W^{1, p}(\Omega)\right),
		\end{equation} where (\ref{dtv}) holds weakly in $\Omega\times I$, which means that for all $t_{1}, t_{2}\in I$ with $t_{1}<t_{2}$, and all non-negative \textit{test functions} 
		\begin{equation}  \label{defvontestsoluq}
			\psi\in W_{loc}^{1, p}\left(I;
			L^{p}(\Omega)\right)\cap L_{loc}^{p}\left(I; W_{0}^{1,
				p}(\Omega)\right),
		\end{equation}
		we have 
		\begin{equation}  \label{defvonweaksolq}
			\left[\int_{\Omega}{u\psi }\right]_{t_{1}}^{t_{2}}+\int_{t_{1}}^{t_{2}}{%
				\int_{\Omega}{-u\partial_{t}\psi+|\nabla u^{\frac{1}{p-1}}|^{p-2}\langle\nabla u^{\frac{1}{p-1}},
					\nabla \psi\rangle}}\leq 0.
		\end{equation}
	\end{definition}
	
	\textit{Weak supersolutions} and \textit{weak solutions} of (\ref{dtv}) are
	defined analogously.
	
	Existence results for weak solutions of (\ref{dtv}) were obtained in \cite{bogelein2018doubly, coulhon2016regularisation,  ishige1996existence, ivanov1997regularity, lindqvist2026lipschitz} in the Euclidean setting and in \cite{surig2026existence} on manifolds.
	
		The next two lemmas can be proved similarly to Lemma 2.6 in \cite{grigor2024finite}.
	\begin{lemma}[Caccioppoli-type inequality]
		\label{Lem1} Let $I$ be an interval in $\mathbb{R}_{+}=[0, \infty)$ and let $u=u\left( x,t\right) $ be a bounded
		non-negative subsolution to \emph{(\ref{dtv})} in $M\times I$. Fix some real $\sigma $
		such that 
		$\sigma \geq \frac{p}{p-1}$. 
		Choose $t_{1},t_{2}\in I$ such that $t_{1}<t_{2}$. Then%
		\begin{equation}
			\left[ \int_{M }u^{\sigma } \right] _{t_{1}}^{t_{2}}+c_{1}\left( \sigma  -1\right)\sigma ^{1-p}
			\int_{M\times [t_{1}, t_{2}]}\left\vert \nabla \left( u^{\sigma/p } \right) \right\vert
			^{p} \leq 0,
			\label{vetacor}
		\end{equation}
		where $c_{1}$ is a constant depending on $p$. 
	\end{lemma}
	
	Let us observe for a later usage that \begin{equation}\label{valpha}u^{\sigma/p}\in L_{loc}^{p}\left(I; W^{1, p}(M)\right).\end{equation}
	Indeed, using $\sigma/p\geq \frac{1}{p-1}$, we get that the function $\Phi(s)=s^{\frac{\sigma(p-1)}{p}}$ is Lipschitz on any bounded interval in $[0, \infty)$. Thus, $u^{\sigma/p}=\Phi(u^{\frac{1}{p-1}})\in W^{1, p}(M)$ and $$\left|\nabla u^{\sigma/p}\right|=\left|\Phi^{\prime}(u^{\frac{1}{p-1}})\nabla u^{\frac{1}{p-1}}\right|\leq C\left|\nabla u^{\frac{1}{p-1}}\right|,$$ whence for any bounded interval $J\subset I$, \begin{equation*}\label{valpha0}\int_{M\times J}u^{\sigma}+\left\vert \nabla \left( u^{\sigma/p } \right) \right\vert^{p}\leq  C^{\prime}\int_{M\times J}u^{\sigma}+\left\vert \nabla  u^{1/(p-1) } \right\vert^{p},\end{equation*} which is finite since $$\int_{M\times J}u^{\sigma}\leq \textnormal{const}~||u||_{L^{\infty}}^{\sigma-\frac{p}{p-1}}\int_{M\times J}u^{\frac{p}{p-1}}$$ and proves (\ref{valpha}).
		
	\begin{lemma}[Lemma 2.9 \cite{Grigoryan2024}]
		\label{monl1}
		Let $u=u\left( x,t\right) $ be a non-negative bounded solution to \emph{(\ref{dtv})} in $M\times I$. If $\sigma\geq 1$, including $\sigma=\infty$, then the function 
		\begin{equation*}
			t\mapsto \left\Vert u(\cdot ,t)\right\Vert _{L^{\sigma}(M)}
		\end{equation*}%
		is monotone decreasing in $I$.
	\end{lemma}

	\section{Relative Faber-Krahn inequality}
	
	\label{SecRFK} 
		
		Using the method from \cite{grigor1994heat}, which uses arguments from \cite{carron1996inegalites}, we obtain the following result. 
		\begin{proposition}\label{FKimpVD}
		Assume that $M$ admits the relative Faber-Krahn inequality (\ref{FKp}). Then $M$ satisfies the volume doubling property (\ref{lemmaausdoubling}).
		\end{proposition}
	
		\begin{proof}
		Fix $x\in M$, $r>0$ and let $w(y)=(r-d(x, y))_+$. Then $|\nabla w|\leq 1$ and $w\geq \frac{1}{2}r$ if $d(x, y)<	\frac{1}{2}r$ and thus, for $0<r\leq R$, $$\frac{\dint_{B(x, R)}\left\vert \nabla w\right\vert ^{p}}{\dint_{B(x, R)}w^{p}}\leq  \frac{\mu(B(x, r))}{\dint_{B(x, \frac{1}{2}r)}w^{p}}\leq \frac{2^p\mu(B(x, r))}{r^{p}\mu(B(x, \frac{1}{2}r))}.$$
		Applying the relative Faber-Krahn inequality (\ref{FKp}) we therefore get $$\frac{\mu(B(x, R))^{\nu}}{R^{p}\mu(B(x, r))^{\nu}}\leq \frac{2^p\mu(B(x, r))}{r^{p}\mu(B(x, \frac{1}{2}r))},$$ that is, $$\mu(B(x, r))\geq cr^{\frac{p}{1+\nu}}\mu(B(x, \frac{1}{2}r))^{\frac{1}{1+\nu}}\frac{\mu(B(x, R))^{\frac{1}{1+\nu}}}{R^{p}}.$$ Iterating this inequality we get, for any $m\geq 1$, $$\mu(B(x, r))\geq c(2^p)^{-\sum_{i=1}^{m}(i-1)\left(\frac{1}{1+\nu}\right)^{i}}\left(r^{p}\frac{\mu(B(x, R))^{\nu}}{R^{p}}\right)^{\sum_{i=1}^{m}\left(\frac{1}{1+\nu}\right)^{i}}\mu(B(x, \frac{r}{2^m})).$$ Sending $m\to \infty$ yields $$\mu(B(x, r))\geq c\left(r^{p}\frac{\mu(B(x, R))^{\nu}}{R^{p}}\right)^{1/\nu},$$ which gives $$\frac{\mu(B(x, R))}{\mu(B(x, r))}\leq C\left(\frac{R}{r}\right)^{p/\nu}$$ and proves (\ref{lemmaausdoubling}) with $N=p/\nu$. 
		\end{proof}
	
	\subsection{Properties of the first eigenfunction}
	
	\begin{definition}
		Let $D$ be a precompact open set in $M$ and let $p>1$. Then we define the first eigenvalue $\lambda_{1, p}(D)$ of the $p$-Laplacian $\Delta _{p}$ by \begin{equation}\label{firsteigdef}
			\lambda_{1, p}(D)=\inf_{0\ne w\in W_0^{1, p}(D)}\frac{\dint_{D}\left\vert \nabla w\right\vert ^{p}}{\dint_{D}|w|^{p}},
		\end{equation}
		where the quotient on the right-hand side of (\ref{firsteigdef}) is the \it{Rayleigh quotient}.
	\end{definition}
	
	\begin{lemma}\label{constreig}
		Let $D$ be a precompact open set in $M$. Then  $\lambda_{1, p}(D)>0$ and there exists an eigenfunction $\varphi\in W_0^{1, p}(D)$ corresponding to $\lambda_{1, p}(D)$, that is, \begin{equation}\label{defeigenfunc}
			\dint_{D}\left\vert \nabla \varphi\right\vert ^{p}=	\lambda_{1, p}(D)\dint_{D}|\varphi|^{p}.
		\end{equation}
		Moreover, $\varphi$ can be chosen non-negative and $\varphi$ satisfies $-\Delta _{p}\varphi=\lambda_{1, p}(D)\varphi^{p-1}$ weakly in $D$, that means, for any $\psi\in W_{0}^{1, p}(D)$, \begin{equation}\label{eigenfunctionprop}\int_D|\nabla \varphi|^{p-2}\langle\nabla \varphi,
			\nabla \psi\rangle=\lambda_{1, p}(D) \int_D \varphi^{p-1}\psi.\end{equation}
	\end{lemma}
	
	\begin{proof}
		Let us choose a sequence $\varphi_j\in W_0^{1, p}(D)\setminus\{0\}$ such that \begin{equation}\label{miniseq}\frac{\dint_{D}\left\vert \nabla \varphi_j\right\vert ^{p}}{\dint_{D}|\varphi_j|^{p}}\to \lambda_{1, p}(D)\quad \textnormal{as}~j\to \infty.\end{equation} By homogeneity of the Rayleigh quotient, we can choose this sequence so that $\dint_{D}|\varphi_j|^{p}=1$ for every $j$. Clearly, the sequence $\{\varphi_j\}$ is bounded in $W_0^{1, p}(D)$ so that, after passing to a subsequence, we have $$\varphi_j\rightharpoonup \varphi~\textnormal{weakly in}~W^{1, p}(D)$$ for some $\varphi\in W_0^{1, p}(D)$. Since the embedding $W_0^{1, p}(D)\hookrightarrow L^{p}(D)$ is compact, we obtain that $\varphi_j\to \varphi~\textnormal{in}~L^{p}(D)$. Hence, $$\dint_{D}|\varphi|^{p}=\lim_{j\to \infty}\dint_{D}|\varphi_j|^{p}=1$$ and in particular, $\varphi\not\equiv0$. Since the function $w\mapsto \dint_{D}\left\vert \nabla w\right\vert ^{p}$ is weakly lower-semicontinuous in $W_0^{1, p}(D)$ by the convexity of the map $\xi\mapsto |\xi|^{p}$, we get by (\ref{miniseq}), $$\dint_{D}\left\vert \nabla \varphi\right\vert ^{p}\leq \liminf_{j\to \infty} \dint_{D}\left\vert \nabla \varphi_j\right\vert ^{p}=\lambda_{1, p}(D).$$ Moreover, $$\lambda_{1, p}(D)\leq \frac{\dint_{D}\left\vert \nabla \varphi\right\vert ^{p}}{\dint_{D}|\varphi|^{p}}=\dint_{D}\left\vert \nabla \varphi\right\vert ^{p}.$$ Therefore, (\ref{defeigenfunc}) follows. In particular, $\lambda_{1, p}(D)>0$. Indeed, otherwise (\ref{defeigenfunc}) would imply $\dint_{D}\left\vert \nabla \varphi\right\vert ^{p}=0$. Hence, $\varphi$ is constant a.e. on every connected component of $D$. Since $\varphi\in W_0^{1, p}(D)$ this yields $\varphi\equiv 0$ a.e. in $D$ contradicting $\dint_{D}|\varphi|^{p}=1$.
		Since $|\nabla \varphi|=|\nabla |\varphi||$ a.e., $\varphi$ can be chosen to be non-negative.
		
		Note that $\varphi$ minimizes the function $w\mapsto \dint_{D}\left\vert \nabla w\right\vert ^{p}$ under the condition that $\dint_{D}|w|^{p}=1$. Hence, by the Lagrange multiplier principle there exists a real $\Lambda$ so that, for any $\psi\in W_{0}^{1, p}(D)$, \begin{equation}\label{lagrange}p\int_D|\nabla \varphi|^{p-2}\langle\nabla \varphi,
			\nabla \psi\rangle=p\Lambda \int_D \varphi^{p-1}\psi.\end{equation} Taking $\psi=\varphi$, we obtain $$\dint_{D}\left\vert \nabla \varphi\right\vert ^{p}=\Lambda\dint_{D}\varphi^{p}=\Lambda.$$ Thus, $\Lambda=\lambda_{1, p}(D)$ and we conclude (\ref{eigenfunctionprop}) from (\ref{lagrange}).
	\end{proof}
	
	\begin{lemma}\label{boundedeigfun}
		Let $D$ be a precompact open set and $\varphi$ be the eigenfunction corresponding to the first eigenvalue $\lambda_{1, p}(D)$ of the $p$-Laplacian $\Delta _{p}$ constructed in Lemma \ref{constreig}. Then $\varphi\in L^{\infty}(D)$.
	\end{lemma}
	
	\begin{proof}
		Since $D$ is precompact the following \textit{Sobolev inequality} in $D$ of order $p\geq 1$ holds: for any non-negative function 
		$w\in W_{0}^{1,p}(D)$,
		\begin{equation}
			\left( \int_{D}w^{p\kappa }\right) ^{1/\kappa }\leq S_{D}\int_{D}\left\vert
			\nabla w\right\vert ^{p},  \label{SBk}
		\end{equation}%
		where $\kappa >1$ is some constant and $S_{D}$ is called the \emph{Sobolev constant} in $D$. In fact, the value of $\kappa $ is independent of $D$ and can be
		chosen as follows:%
		\begin{equation}
			\kappa =\left\{ 
			\begin{array}{ll}
				\dfrac{n}{n-p}, & \text{if }n>p, \\ 
				\text{any number}>1, & \text{if }n\leq p.%
			\end{array}%
			\right.  \label{k}
		\end{equation}
		
		Let us show inductively that $\varphi\in L^{q_j}(D)$, $q_j=p\kappa^{j}$, for any $j\geq 0$. 
		Since $\varphi\in W_{0}^{1,p}(D)$, the case $j=0$ trivially holds.
		
		Let us now assume that $\varphi\in L^{q}(D)$ for some $q\geq p$.
		For $k>0$ let $\varphi_k=\min(k, \varphi)$ and consider the test function $\psi_k=\varphi_k^{q-p+1}$. Since $q\geq p$, the function $s\mapsto \min(k, s)^{q-p+1}$ is Lipschitz and vanishes at $s=0$. Hence, $ \psi_k\in W_{0}^{1,p}(D)$ and $$\nabla \psi_k=(q-p+1)\varphi_k^{q-p}\nabla \varphi_k=(q-p+1)\varphi_k^{q-p}1_{\{\varphi<k\}}\nabla \varphi.$$
		Testing with this function in (\ref{eigenfunctionprop}) we obtain $$(q-p+1)\int_{\{\varphi<k\}}\varphi^{q-p}|\nabla \varphi|^p=\lambda_{1, p}(D)\int_D \varphi^{p-1}\varphi_k^{q-p+1}.$$ Since $\varphi_k\leq \varphi$, we get \begin{equation}\label{testingvarph}(q-p+1)\int_{\{\varphi<k\}}\varphi^{q-p}|\nabla \varphi|^p\leq \lambda_{1, p}(D)\int_D \varphi^{q},\end{equation} where the right-hand side is finite by induction hypothesis. Let now $w_k=\varphi_k^{q/p}$. Then again $w_k\in W_{0}^{1,p}(D)$ and $$\nabla w_k=\frac qp\varphi_k^{q/p-1}1_{\{\varphi<k\}}\nabla \varphi.$$ Hence, $$\int_D|\nabla w_k|^p=\left(\frac qp\right)^{p}\int_{\{\varphi<k\}}\varphi^{q-p}|\nabla \varphi|^p.$$
		Together with (\ref{testingvarph}) this yields $$\int_D|\nabla w_k|^p\leq\lambda_{1, p}(D)\frac{\left( q/p\right)^{p}}{q-p+1}\int_D \varphi^{q}.$$
		Combining this with the Sobolev inequality (\ref{SBk}) applied to $w_k$ we obtain 
		$$\left( \int_{D}\varphi_k^{q\kappa }\right) ^{1/\kappa }\leq S_D \lambda_{1, p}(D)\frac{\left( q/p\right)^{p}}{q-p+1}\int_D \varphi^{q}.$$ Since $\varphi_k\to \varphi$ pointwise as $k\to \infty$, we deduce by the monotone convergence theorem that $$\int_{D}\varphi_k^{q\kappa }\to \int_{D}\varphi^{q\kappa }.$$ Using that $\frac{\left( q/p\right)^{p}}{q-p+1}\leq C_pq^{p-1}$ it follows that $$\left( \int_{D}\varphi^{q\kappa }\right) ^{1/\kappa }\leq C_pS_D \lambda_{1, p}(D)q^{p-1}\int_D \varphi^{q},$$ which implies \begin{equation}\label{beforeite}||\varphi||_{L^{q\kappa}(D)}\leq \left(C_pS_D \lambda_{1, p}(D)q^{p-1}\right)^{1/q}||\varphi||_{L^{q}(D)}.\end{equation}
		Iterating this inequality we obtain with $q_j=p\kappa^{j}$, for any $m\geq 1$, $$||\varphi||_{L^{q_{m+1}}(D)}\leq \prod_{j=0}^{m}\left(C_pS_D \lambda_{1, p}(D)q_j^{p-1}\right)^{1/q_j}||\varphi||_{L^{p}(D)}.$$ Since $\sum_{j=0}^{\infty}\frac{1}{q_j}<\infty$ and $\sum_{j=0}^{\infty}\frac{\log q_j}{q_j}<\infty$ as $\kappa>1$, it follows that $$||\varphi||_{L^{q_{m+1}}(D)}\leq C\left(S_D \lambda_{1, p}(D)\right)^{\sum_{j=0}^{\infty}1/q_j}||\varphi||_{L^{p}(D)},$$ where $C$ is independent of $m$.
		Thus, sending $m\to \infty$, we conclude $$||\varphi||_{L^{\infty}(D)}\leq C\left(S_D \lambda_{1, p}(D)\right)^{\frac{\kappa}{p(\kappa-1)}}||\varphi||_{L^{p}(D)},$$ which finishes the proof. 
	\end{proof}
	
	\begin{lemma}\label{lemmaeigen}
		Let $D$ be a precompact open set and $\lambda_{1, p}(D)$ be the first eigenvalue of the $p$-Laplacian $\Delta _{p}$ with the corresponding eigenfunction $\varphi$.
		Let $\widetilde{\varphi}$ be the zero extension of $\varphi$. Then $v(x, t)=e^{-\lambda_{1, p}(D)t}\widetilde{\varphi}(x)^{p-1}$ is a weak subsolution of (\ref{dtv}) in $M\times [0, \infty)$.
	\end{lemma}
	
	\begin{proof}
		Since $\varphi\in W_{0}^{1, p}(D)$ is the Dirichlet eigenfunction of the $p$-Laplacian $\Delta _{p}$ corresponding to the first eigenvalue $\lambda=\lambda_{1, p}(D)$ we have, for any $\zeta\in W_{0}^{1, p}(D)$, \begin{equation}\label{eigenfunction}\int_D|\nabla \varphi|^{p-2}\langle\nabla \varphi,
			\nabla \zeta\rangle=\lambda \int_D \varphi^{p-1}\zeta.\end{equation}
		Let $\psi\in W^{1, p}(M)$ be non-negative and define, for any $\varepsilon>0$, $\theta_\varepsilon(s)=\min(1, \frac{s_+}{\varepsilon})$ and $\psi_\varepsilon=\psi \theta_\varepsilon(\varphi)$. By approximation we can also assume that $\psi$ is bounded. Since $\theta_\varepsilon(s)$ is Lipschitz and $\theta_\varepsilon(0)=0$, we have $\theta_\varepsilon(\varphi)\in W_{0}^{1, p}(D)$ and thus, $\psi_\varepsilon\in W_{0}^{1, p}(D)$. Hence, we can apply (\ref{eigenfunction}) with $\psi_\varepsilon\in W_{0}^{1, p}(D)$ and obtain $$\int_D|\nabla \varphi|^{p-2}\langle\nabla \varphi,
		\nabla (\psi \theta_\varepsilon(\varphi))\rangle=\lambda \int_D \varphi^{p-1}\psi \theta_\varepsilon(\varphi).$$
		It follows that $$\int_D\theta_\varepsilon(\varphi)|\nabla \varphi|^{p-2}\langle\nabla \varphi,
		\nabla \psi \rangle+\int_D\psi\theta_\varepsilon(\varphi)^{\prime}|\nabla \varphi|^{p} =\lambda \int_D \varphi^{p-1}\psi \theta_\varepsilon(\varphi).$$ Since the second term on the left-hand side is non-negative we obtain $$\int_D\theta_\varepsilon(\varphi)|\nabla \varphi|^{p-2}\langle\nabla \varphi,
		\nabla \psi \rangle \leq\lambda \int_D \varphi^{p-1}\psi \theta_\varepsilon(\varphi).$$ For $\varepsilon \to 0$ we have $\theta_\varepsilon(\varphi)\to 1_{\{\varphi>0\}}$ so that by the dominated convergence theorem $$\int_D|\nabla \varphi|^{p-2}\langle\nabla \varphi,
		\nabla \psi \rangle \leq\lambda \int_D \varphi^{p-1}\psi .$$ Therefore, for any non-negative $\psi \in W^{1, p}(M)$, \begin{equation}\label{zerextarg}\int_M|\nabla \widetilde{\varphi}|^{p-2}\langle\nabla \widetilde{\varphi},
			\nabla \psi \rangle \leq\lambda \int_M \widetilde{\varphi}^{p-1}\psi .\end{equation}
		We have the identity $$e^{-\lambda t}|\nabla \widetilde{\varphi}|^{p-2}\nabla \widetilde{\varphi}=|\nabla v^{1/(p-1)}|^{p-2}\nabla v^{1/(p-1)}.$$
		Hence, multiplying both sides of (\ref{zerextarg}) with $e^{-\lambda t}$ and integrating over $[t_1, t_2]$ yields $$\int_{t_{1}}^{t_{2}}\int_{M}|\nabla v^{1/(p-1)}|^{p-2}\langle \nabla v^{1/(p-1)}, \nabla \psi \rangle\leq\lambda \int_{t_{1}}^{t_{2}}\int_{M}v\psi.$$
		Using $\lambda v=-\partial_t v$ and partial integration, we obtain for any $\psi$ as in (\ref{defvontestsoluq}), $$\left[\int_{M}{v\psi }\right]_{t_{1}}^{t_{2}}+\int_{t_{1}}^{t_{2}}\int_{M}|\nabla v^{1/(p-1)}|^{p-2}\langle \nabla v^{1/(p-1)}, \nabla \psi \rangle\leq \int_{t_{1}}^{t_{2}}
		\int_{M}v\partial_{t}\psi,$$ which proves that $v$ is a weak subsolution of (\ref{dtv}) in $M\times [0, \infty)$.
	\end{proof}

	\subsection{Sub-Gaussian upper estimate}

The following result proves the implication $\eqref{FKp}\Leftarrow \eqref{lemmaausdoubling}+\eqref{uppergaussint}$ from Theorem \ref{mainthmint} from the introduction.
		
	\begin{theorem}\label{mainthmFK}
	Assume that $M$ admits the sub-Gaussian upper estimate (\ref{uppergaussint}) and the volume doubling property (\ref{lemmaausdoubling}). Then $M$ satisfies the relative Faber-Krahn inequality (\ref{FKp}).
	\end{theorem}

\begin{remark}
The implication that the relative Faber-Krahn inequality (\ref{FKp}) implies the sub-Gaussian upper estimate (\ref{uppergaussint}) was proved in Theorem 1.1 in \cite{surig2024sharp}.
\end{remark}

\begin{proof}
Let $D$ be a precompact open set in $M$ and let $\lambda_{1, p}(D)$ be the first eigenvalue of the $p$-Laplace with the corresponding eigenfunction $\varphi\in W_0^{1, p}(D)$. Hence, for any $\psi\in W_{0}^{1, p}(D)$,$$\int_D|\nabla \varphi|^{p-2}\langle\nabla \varphi,
\nabla \psi\rangle=\lambda_{1, p}(D)\int_D \varphi^{p-1}\psi.$$  Since, for any constant $c>0$, \begin{equation*}\int_D|\nabla (c\varphi)|^{p-2}\langle\nabla (c\varphi),	\nabla \psi\rangle=\lambda_{1, p}(D)\int_D (c\varphi)^{p-1}\psi,\end{equation*} we can assume that $\int_{D}\varphi^{p}=1$. By Lemma \ref{constreig} we have  $$\lambda_{1, p}(D)=\inf_{0\ne w\in W_0^{1, p}(D)}\frac{\dint_{D}\left\vert \nabla w\right\vert ^{p}}{\dint_{D}w^{p}}>0.$$ 
We need to show that \begin{equation}\label{rFKp}\lambda_{1, p}(D)\geq \frac{c}{R^{p}}\left(\frac{\mu(B(x_0, R))}{\mu(D)}\right)^{\nu},\quad D\subset B(x_0, R).\end{equation}
Let us extend $\varphi$ by zero to $M$. From Lemma \ref{boundedeigfun} and Lemma \ref{lemmaeigen} we know that $v(x, t)=e^{-\lambda_{1, p}(D)t}\varphi(x)^{p-1}$ is a bounded weak subsolution of (\ref{dtv}) in $M\times [0, \infty)$. In particular, $v$ satisfies $v^{\frac{1}{p-1}}(\cdot, t)\in W_0^{1, p}(M)$ and has initial function $v(x, 0)=\varphi^{p-1}(x)\in L^{1}(M)\cap L^{\infty}(M)$.
Hence, the sub-Gaussian upper estimate (\ref{uppergaussint}) yields, for all $x\in M$ and all $t>0$, \begin{equation}\label{useofgaus}||v(\cdot,t)||_{L^{\infty}\left(B(x, \frac{1}{2}t^{1/p})\right)}\leq\frac{C||\varphi^{p-1}||_{L^{1}(D)}}{\mu(B(x, t^{1/p}))}\exp\left(-c\left(\frac{d(x,A)}{t^{1/p}}\right)^{\frac{p}{p-1}}\right),\end{equation} where $A=\textnormal{supp}~\varphi^{p-1}$.

Let us set $$E=\{x\in D: \varphi(x)>\frac12\mu(D)^{-1/p}\}.$$
The set $E$ has positive measure as otherwise, $$1=\int_D\varphi^p\leq \frac{1}{2^p}\mu(D)^{-1}\mu(D)<1.$$
Therefore, we can choose $x_D\in E$ such that $\mu(E\cap B(x_D, r))>0$ for every $r>0$. Thus, $$||v(\cdot,t)||_{L^{\infty}\left(B(x_D, \frac{1}{2}t^{1/p})\right)}\geq e^{-\lambda_{1, p}(D)t}\frac{1}{2^{p-1}}\mu(D)^{-(p-1)/p}.$$
Moreover, by H\"older's inequality, $$\int_D\varphi^{p-1}\leq \left(\int_D\varphi^p\right)^{(p-1)/p}\mu(D)^{1/p}=\mu(D)^{1/p}.$$
Hence, we obtain from (\ref{useofgaus}), for any $t>0$, \begin{equation}\label{uppere}e^{-\lambda_{1, p}(D)t}\leq \frac{C\mu(D)}{\mu(B(x_D, t^{1/p}))}.\end{equation}

Let $D\subset B(x_0, R)$. Then we have $d(x_D, x_0)<R$ so that (\ref{lemmaausdoubling}) yields, for all $0<r\leq R$, \begin{equation}\label{compom}
\mu(B(x_D, r))\geq c \left(\frac{r}{R}\right)^{N}\mu(B(x_0, R)).
\end{equation} 
For $r=t^{1/p}$, we deduce by (\ref{uppere}) and (\ref{compom}), for $0<r\leq R$, \begin{equation}\label{largeR}e^{-\lambda_{1, p}(D)r^{p}}\leq C\left(\frac{R}{r}\right)^{N}\frac{\mu(D)}{\mu(B(x_0, R))}.\end{equation}
Let us assume that $\frac{C\mu(D)}{\mu(B(x_0, R))}\leq \varepsilon<1$. In this case, we take $r=C^{1/N}\varepsilon ^{-1/N}R\left(\frac{\mu(D)}{\mu(B(x_0, R))}\right)^{1/N}.$ Then we have $r\leq R$, whence by (\ref{largeR}), $$e^{-\lambda_{1, p}(D)r^{p}}\leq \varepsilon.$$ Hence, $\lambda_{1, p}(D)r^{p}\geq \log \frac{1}{\varepsilon}$ and thus by our choice of $r$, $$\lambda_{1, p}(D)\geq \frac{\varepsilon^{p/N}\log \frac{1}{\varepsilon}}{C^{p/N}R^{p}}\left(\frac{\mu(B(x_0, R))}{\mu(D)}\right)^{p/N},$$ which proves (\ref{rFKp}) with $\nu=p/N$.

Hence, it remains to consider the case when $\frac{C\mu(D)}{\mu(B(x_0, R))}\geq \varepsilon$. In this case let us set $t^{1/p}=AR$, where $A>1$ is to be chosen. It is proved in \cite{grigor2014upper} that the doubling property (\ref{lemmaausdoubling}) implies that there exist positive constants $c, N_0$ such that, for all $x\in M$ and all $r>0$, \begin{equation}\label{revdoubling}\mu(B(x, Ar))\geq c A^{N_0}\mu(B(x, r)).\end{equation} Hence, by the reverse doubling (\ref{revdoubling}), $d(x_D, x_0)<R$ and the volume doubling (\ref{lemmaausdoubling}), $$\mu(B(x_D, AR))\geq cA^{N_0}\mu(B(x_0, R)).$$ Then (\ref{uppere}) with $t=(AR)^{p}$ yields $$e^{-\lambda_{1, p}(D)(AR)^{p}}\leq \frac{C\mu(D)}{\mu(B(x_D, AR))}\leq CA^{-N_0}.$$ Choosing $A$ large enough so that $CA^{-N_0}\leq 1/2$ we get $e^{-\lambda_{1, p}(D)(AR)^{p}}\leq 1/2$ and therefore, $$\lambda_{1, p}(D)\geq \frac{\log 2}{A^{p}R^{p}}\geq \frac{\log 2\varepsilon^{p/N}}{A^{p}C^{p/N}R^{p}}\left(\frac{\mu(B(x_0, R))}{\mu(D)}\right)^{p/N},$$ which proves also (\ref{rFKp}) in this case with $\nu=p/N$. 
\end{proof}

\section{Uniform Faber-Krahn inequality}

\label{Secini} 

\begin{lemma}\label{logsobolev}
	Suppose that $M$ admits a uniform $p$-Faber-Krahn inequality (\ref{uFKp}). Then, for all precompact open $\Omega\subset M$, any non-negative $v\in  C_{0}^{\infty}(\Omega)$ and any $0<\varepsilon<1$, we have \begin{equation}\label{lemFk}\int_{\Omega}{|\nabla v|^{p}d\mu}\geq \Lambda_{p}\left(\left(\frac{p}{\varepsilon}\right)^{p-1}\frac{\left(\int_{\Omega}v^{p-1}\right)^{p}}{\left(\int_{\Omega}v^{p}\right)^{p-1}}\right)(1-\varepsilon)\int_{\Omega}v^{p}.\end{equation}
\end{lemma}

\begin{proof}
	For any $s>0$ it holds that \begin{equation}\label{ineqforFk}v^{p}\leq (v-s)_{+}^{p}+psv^{p-1}.\end{equation}		
	Let us set $\Omega_{s}=\{v>s\}$. Then integrating (\ref{ineqforFk}) we obtain $$\int_{\Omega}v^{p}\leq \int_{\Omega_{s}}(v-s)_{+}^{p}+ps \int_{\Omega}v^{p-1}.$$
	Hence, we obtain from the uniform $p$-Faber-Krahn inequality (\ref{uFKp}) in $\Omega_{s}$, $$\int_{\Omega}v^{p}\leq \frac{\int_{\Omega}{|\nabla v|^{p}}}{\Lambda_{p}(\mu(\Omega_{s}))}+ps \int_{\Omega}v^{p-1}.$$
	Since $\mu(\Omega_{s})\leq \frac{1}{s^{p-1}}		\int_{\Omega}v^{p-1}$ we deduce $$\int_{\Omega}{|\nabla v|^{p}}\geq \Lambda_{p}\left(\frac{1}{s^{p-1}}		\int_{\Omega}v^{p-1}\right)\left(\int_{\Omega}v^{p}-ps \int_{\Omega}v^{p-1}\right).$$
	Choosing $$s=\frac{\varepsilon}{p}	\frac{\int_{\Omega}v^{p}}{\int_{\Omega}v^{p-1}}$$ we therefore obtain (\ref{lemFk}).\end{proof}

\begin{lemma}\label{thmfinexinf}
	Let $u$ be a non-negative bounded subsolution of (\ref{dtv}) in $M\times \mathbb{R}_{+}$ with $u(\cdot, 0)=u_{0}\in L^{1}(M)$. Suppose that $M$ admits a uniform $p$-Faber-Krahn inequality (\ref{uFKp}) with Faber-Krahn function $\Lambda_{p}$. Then, for any $\sigma\geq \frac{p}{p-1}$ and for any $0<\varepsilon<1$,  \begin{equation}\label{LqFkup}||u(t)||_{L^{\sigma}(M)}\leq \left(\frac{p}{\varepsilon}\right)^{\frac{(p-1)(\sigma-1)}{\sigma}} \frac{||u_{0}||_{L^{1}(M)}}{\gamma\left(c\sigma^{-(p-1)}(1-\varepsilon)t\right)^{\frac{\sigma-1}{\sigma}}} ,\end{equation} where $\gamma$ is defined by \begin{equation}\label{gamma}t=\int_{0}^{\gamma(t)}\frac{dv}{\Lambda_{p}(v)v}.\end{equation}
\end{lemma}

\begin{proof}
	Set $$\Phi(t)=\int_{M }u^{\sigma }(t).$$ From (\ref{vetacor}) we have, for any $t>0$ and any $h>0$, $$ \Phi(t+h)-\Phi(t) \leq-c_{1}\left( \sigma  -1\right)\sigma ^{1-p}
	\int_{t}^{t+h}\int_{M}\left\vert \nabla \left( u^{\sigma/p } \right) \right\vert
	^{p}.$$ Hence, dividing both sides by $h$ and sending $h\to 0$, we deduce, for a.e. $t>0$, \begin{equation}\label{diffgrad}\frac{d}{dt}\Phi(t)\leq-c_{1}\left( \sigma  -1\right)\sigma ^{1-p}
	\int_{M}\left\vert \nabla \left( u^{\sigma/p } \right) \right\vert
	^{p}.\end{equation} Let us combine this with the estimate (\ref{lemFk}) applied to $v_R=u^{\sigma/p}\eta_R\in W_{0}^{1, p}(B_{2R})$ (cf. (\ref{valpha})) where $\eta_R$ is a cut-off function of some precompact geodesic ball $B_R$ in $B_{2R}$. Indeed, we have, since $\sigma\geq \frac{p}{p-1}$, $$\int_{M}{u^{\sigma}|\nabla \eta_R|^{p}}\leq \frac{C}{R^{p}} \int_{M}{u^{\sigma}}\to 0\quad\textnormal{as}~R\to \infty$$ and by dominated convergence $$\int_{M}{|\eta_R\nabla u^{\sigma/p}|^{p}}\to \int_{M}{|\nabla u^{\sigma/p}|^{p}} \quad\textnormal{as}~R\to \infty.$$ Similarly, we see that $$\int_{M}v_R^{p-1}\to \int_{M}u^{\sigma\frac{p-1}{p}}\quad \textnormal{and}\quad \int_{M}v_R^{p}\to \int_{M}u^{\sigma}.$$ Hence, with the monotonicity of $\Lambda_p$ we obtain from (\ref{diffgrad}), for any $0<\varepsilon<1$, $$\frac{d}{dt}\Phi(t)\leq-c_{1}\left( \sigma  -1\right)\sigma ^{1-p} \Lambda_{p}\left(\left(\frac{p}{\varepsilon}\right)^{p-1}\frac{\left(\int_{M}u^{\sigma\frac{p-1}{p}}\right)^{p}}{\left(\int_{M}u^{\sigma}\right)^{p-1}}\right)(1-\varepsilon)\int_{M}u^{\sigma}.$$
	We have from Hölder's inequality  $$\int_{M}u^{\sigma\frac{p-1}{p}}\leq \left(\int_{M}u^{\sigma}\right)^{\frac{\sigma\frac{p-1}{p}-1}{\sigma-1}} \left(\int_{M}u\right)^{\frac{\sigma-\sigma\frac{p-1}{p}}{\sigma-1}}.$$ 
	Therefore, using also Lemma \ref{monl1}, $$\frac{\left(\int_{M}u^{\sigma\frac{p-1}{p}}\right)^{p}}{\left(\int_{M}u^{\sigma}\right)^{p-1}}\leq \frac{\left(\int_{M}u^{\sigma}\right)^{\frac{\sigma(p-1)-p}{\sigma-1}}\left(\int_{M}u_{0}\right)^{\frac{\sigma p-\sigma(p-1)}{\sigma-1}}}{\left(\int_{M}u^{\sigma}\right)^{p-1}}=\frac{\left(\int_{M}u_{0}\right)^{\frac{\sigma}{\sigma-1}}}{\left(\int_{M}u^{\sigma}\right)^{\frac{1}{\sigma-1}}}.$$
	Hence, $$\frac{d}{dt}\Phi(t)\leq-c_{1}\left( \sigma  -1\right)\sigma ^{1-p} \Lambda_{p}\left(\left(\frac{p}{\varepsilon}\right)^{p-1}\frac{\left(\int_{M}u_{0}\right)^{\frac{\sigma}{\sigma-1}}}{\Phi(t)^{\frac{1}{\sigma-1}}}\right)(1-\varepsilon) \Phi(t).$$ Integrating this inequality, we obtain $$\int_{\Phi(t_{0})}^{\Phi(t)}\frac{d\Phi}{\Lambda_{p}\left(\left(\frac{p}{\varepsilon}\right)^{p-1}\frac{\left(\int_{M}u_{0}\right)^{\frac{\sigma}{\sigma-1}}}{\Phi(t)^{\frac{1}{\sigma-1}}}\right)\Phi(t)}\leq -c_{1}\left( \sigma  -1\right)\sigma ^{1-p}(1-\varepsilon)(t-t_{0}) .$$ Making the change $w=\left(\frac{p}{\varepsilon}\right)^{p-1}\frac{\left(\int_{M}u_{0}\right)^{\frac{\sigma}{\sigma-1}}}{\Phi(t)^{\frac{1}{\sigma-1}}},$ it follows that
	\begin{align*}\int_{0}^{w(t)}\frac{dw}{\Lambda_{p}\left(w\right)w}\geq \int_{w(t_{0})}^{w(t)}\frac{dw}{\Lambda_{p}\left(w\right)w}\geq c_{1}\sigma ^{1-p}(1-\varepsilon)(t-t_{0}).\end{align*}
	Sending $t_{0}\to 0$ and using the definition of $\gamma(t)$, we deduce $$w(t)\geq \gamma\left(c_{1}\sigma ^{1-p}(1-\varepsilon)t\right),$$ which yields (\ref{LqFkup}).\end{proof}

The following result contains Theorem \ref{ThmSobint} from the introduction.

\begin{theorem}\label{uniFKthm}
Let $u$ be a non-negative bounded subsolution of (\ref{dtv}) in $M\times \mathbb{R}_{+}$ with $u(\cdot, 0)=u_{0}\in L^{1}(M)$. Suppose that $M$ admits a uniform $p$-Faber-Krahn inequality (\ref{uFKp}) with Faber-Krahn function $\Lambda_{p}$. Fix some $\sigma\geq \frac{p}{p-1}$. Then, for all $x\in M$ and all $t>0$, \begin{equation}\label{Linfupper}||u(\cdot,t)||_{L^{\infty}\left(B(x,\frac{1}{2}t^{1/p})\right)}\leq \frac{CS_B^{1/(\nu \sigma)}||u_{0}||_{L^{1}(M)}}{t^{1/(\nu\sigma)}\gamma(c\sigma^{-(p-1)}t)^{(\sigma-1)/\sigma}}\exp\left(-c\left(\frac{d(x,A)}{t^{1/p}}\right)^{\frac{p}{p-1}}\right),\end{equation} where $\gamma$ is defined by (\ref{gamma}) and $S_B$ denotes the Sobolev constant (cf. (\ref{SBk})) in $B\left( x,t^{1/p}\right)$, $\nu$ is given by (\ref{nu}) and $c, C$ depend on $p$ and $\nu$.
\end{theorem}

\begin{proof}
It follows from (\ref{LqFkup}) that 	
\begin{equation}\label{condforgamma}\int_{M}u^{\sigma}(\cdot, t)\leq \frac{C_{\sigma}||u_{0}||^{\sigma}_{L^{1}(M)}}{\gamma\left(c\sigma^{-(p-1)}t\right)^{\sigma-1}},\end{equation} where $\gamma$ is defined by (\ref{gamma}). Then it follows from Lemma 5.1 in \cite{surig2024sharp} that, for all $x\in M$ and all $t>0$, \begin{align}\label{uxTupper}||u(\cdot,t)||_{L^{\infty}\left(B(x,\frac{1}{2}t^{1/p})\right)}\leq&\left(\frac{C_{\sigma}S_{B}^{1/\nu}||u_{0}||^{\sigma}_{L^{1}(M)}}{t^{1/\nu}\gamma(c\sigma^{-(p-1)}t)^{\sigma-1}}\right)^{\frac{1}{\sigma}}\exp\left(-c\left(\frac{d(x,A)}{t^{1/p}}\right)^{\frac{p}{p-1}}\right),\end{align} which proves (\ref{Linfupper}).
\end{proof}

\begin{remark}\label{remuniformSb}
For example, if $n>p$ and $M$ satisfies the $p$-Faber-Krahn inequality (\ref{uFKp}) with Faber-Krahn function $\Lambda_{p}(v)\geq cv^{-p/n}$, then $S_B\leq \textnormal{const}$ for all balls, that is, in this case the constant $S_B$ in (\ref{Linfupper}) is uniformly bounded in $x$ and $t$.
\end{remark}

In the following examples let us assume that $n>p$.

\begin{example}\label{ex1}
Suppose that, for all large $v$, \begin{equation}\label{lowerfk}\Lambda_p(v)\simeq v^{-p/n}.
\end{equation} Then
by the definition of $\gamma$ and using (\ref{lowerfk}) we obtain, for large $t$, $$t=\int_{0}^{\gamma(t)}\frac{dv}{\Lambda_{p}(v)v}\simeq \gamma(t)^{p/n}.$$ Thus, $\gamma(t)\simeq t^{n/p},$ which implies by Remark \ref{remuniformSb} and (\ref{Linfupper}), for large $t$,
\begin{equation}\label{Linfuppercor}||u(\cdot,t)||_{L^{\infty}\left(B(x,\frac{1}{2}t^{1/p})\right)}\leq \frac{C||u_{0}||_{L^{1}(M)}}{t^{n/p}}\exp\left(-c\left(\frac{d(x,A)}{t^{1/p}}\right)^{\frac{p}{p-1}}\right).\end{equation}
\end{example}

\begin{example}\label{ex2}
Suppose that, for all large $v$, \begin{equation}\label{lowerfk2}\Lambda_p(v)\simeq (\log v)^{-\alpha},
\end{equation} where $\alpha\geq0$. For example, when $M$ is a \textit{hyperbolic space} $\mathbb{H}^{n}$ we have $\alpha=0$. For $\alpha>0$ see \cite{coulhon1993isoperimetrie}. Then, for large $t$, $$t=\int_{0}^{\gamma(t)}\frac{dv}{\Lambda_{p}(v)v}\simeq \log (\gamma (t))^{\alpha+1}.$$
Therefore, $$\gamma(t)\simeq \exp(c t^{1/(\alpha+1)}),$$ and we obtain from Remark \ref{remuniformSb} and (\ref{Linfupper}), for large $t$,
$$||u(\cdot,t)||_{L^{\infty}\left(B(x,\frac{1}{2}t^{1/p})\right)}\leq \frac{C||u_{0}||_{L^{1}(M)}}{t^{n/(p\sigma)}\exp(c_\sigma t^{1/(\alpha+1)})}\exp\left(-c\left(\frac{d(x,A)}{t^{1/p}}\right)^{\frac{p}{p-1}}\right).$$
\end{example}

We say that a function $f$ has \textit{a polynomial decay} if there is $c>0$ such that, for all $t>0$, \begin{equation}\label{polydecay}f(2t)\geq cf(t).
\end{equation}

The next result is in some sense a converse statement to Theorem \ref{uniFKthm}.

\begin{proposition}\label{convuFK}
Fix some $\sigma\geq \frac{p}{p-1}$ and set $$F_{\sigma}(t)=t^{n/(p\sigma)}\gamma(c\sigma^{-(p-1)}t)^{(\sigma-1)/\sigma},$$ where $\gamma$ is a positive, unbounded, increasing $C^{1}$-function such that $\Lambda_p(\gamma(t)):=\frac{\gamma^{\prime}(t)}{\gamma(t)}$ is non-increasing. Assume that $f_{\sigma}(t):=(\log F_{\sigma})^{\prime}(t)$ has a polynomial decay and satisfies \begin{equation}\label{estFsig}f_{\sigma}(t)\simeq \Lambda_p(F_\sigma(t)).\end{equation} Suppose that any non-negative bounded subsolution $u$ of (\ref{dtv}) in $M\times [0, \infty)$ with initial function $u_0\in L^{1}(M)\cap L^{\infty}(M)$ satisfies, for all $x\in M$ and all $t>0$, \begin{equation}\label{Linfupperintconv}||u(\cdot,t)||_{L^{\infty}\left(B(x,\frac{1}{2}t^{1/p})\right)}\leq \frac{C||u_{0}||_{L^{1}(M)}}{F_{\sigma}(t)}\exp\left(-c\left(\frac{d(x,A)}{t^{1/p}}\right)^{\frac{p}{p-1}}\right),\end{equation} where $c, C$ are positive constants. Then, for any precompact open set $\Omega\subset M$, \begin{equation}\label{uFKpoly}
\lambda_{1, p}(\Omega)\geq c\Lambda_p(C\mu(\Omega)).
\end{equation} 
\end{proposition}

\begin{proof}
Arguing as in the proof of Theorem \ref{mainthmFK} we get, for all $t>0$, 
$$ e^{-\lambda_{1, p}(\Omega)t}\leq C\frac{\mu(\Omega)}{F_{\sigma}(t)}.$$ 
Let $K>C$ and choose some $\tau>0$ such that $F_{\sigma}(\tau)=K\mu(\Omega)$. Then $e^{-2\lambda_{1, p}(\Omega)\tau}\leq \frac{CF_{\sigma}(\tau)}{KF_{\sigma}(2\tau)}$ and thus $$\lambda_{1, p}(\Omega)\geq\frac{1}{2\tau}\log\frac{F_{\sigma}(\tau)}{F_{\sigma}(2\tau)}.$$ By the mean value theorem there exists $\theta\in (\tau, 2\tau)$ such that $\frac{\log F_{\sigma}(2\tau)-\log F_{\sigma}(\tau)}{\tau}=f_{\sigma}(\theta)$ and therefore $$\lambda_{1, p}(\Omega)\geq\frac{1}{2}f_{\sigma}(\theta).$$ Using the polynomial decay of $f_{\sigma}$ and (\ref{estFsig}) it follows from our choice of $\tau$ that $$\lambda_{1, p}(\Omega)\geq c\Lambda_p(K\mu(\Omega)),$$ what was to be proved.
\end{proof}

\begin{example}
Let $\gamma(t)=t^{n/p}$, $n>p$, for large $t$. Then $\Lambda_p(\gamma(t))=\frac{\gamma^{\prime}(t)}{\gamma(t)}\simeq t^{-1}$ and thus $\Lambda_p(v)=v^{-p/n}$ for large $v$. Hence, $f_\sigma(t)\simeq t^{-1}$, which clearly has a polynomial decay. Finally, $\Lambda_p(F_\sigma(t))\simeq (t^{n/p})^{-p/n}\simeq t^{-1}$ which yields (\ref{estFsig}).
\end{example}

\begin{example}
Let $\gamma(t)=\exp(ct^{1/(\alpha+1)})$, $\alpha\geq0$, for large $t$. Then $\Lambda_p(\gamma(t))=\frac{\gamma^{\prime}(t)}{\gamma(t)}\simeq t^{-\alpha/(\alpha+1)}$ and, since $\log \gamma(t)\simeq t^{1/(\alpha+1)}$ we have $\Lambda_p(v)=(\log v)^{-\alpha}$ for large $v$. Further, $f_\sigma(t)\simeq t^{-\alpha/(\alpha+1)}$ if $\alpha>0$ and $f_\sigma(t)\simeq 1$ if $\alpha=0$, which clearly has a polynomial decay. Finally, $\Lambda_p(F_\sigma(t))\simeq (\log F_\sigma(t))^{-\alpha}\simeq t^{-\alpha/(\alpha+1)}$ which also implies (\ref{estFsig}).
\end{example}

\begin{da} \normalfont
	This article has no associated data.
\end{da}

\begin{coi} \normalfont
	The author confirms that he does not have actual or potential conflict of interest in
	relation to this publication.
\end{coi}
	
	\bibliographystyle{abbrv}
	\bibliography{librarycacc}
	
	\emph{Universit\"{a}t Bielefeld, Fakult\"{a}t f\"{u}r Mathematik, Postfach
		100131, D-33501, Bielefeld, Germany}
	
	\texttt{philipp.suerig@uni-bielefeld.de}
\end{document}